\documentclass[12pt]{amsart}
\usepackage{amscd}
\def\NZQ{\Bbb}               
\def\NN{{\NZQ N}}

\def\ZZ{{\NZQ Z}}

\def\frk{\frak}               

\def\Phi{{\frk n}}
\def\Phi{{\frk N}}
\def\opn#1#2{\def#1{\operatorname{#2}}} 
\opn\chara{char} \opn\length{\ell} \opn\pd{pd} \opn\rk{rk}
\opn\projdim{proj\,dim} \opn\injdim{inj\,dim} \opn\rank{rank}
\opn\depth{depth} \opn\sdepth{sdepth} \opn\fdepth{fdepth}
\opn\grade{grade} \opn\height{height} \opn\embdim{emb\,dim}
\opn\codim{codim}  \opn\min{min} \opn\max{max}

\opn\Tr{Tr} \opn\bigrank{big\,rank}
\opn\superheight{superheight}\opn\lcm{lcm}
\opn\trdeg{tr\,deg}
\opn\reg{reg} \opn\lreg{lreg} \opn\ini{in} \opn\lpd{lpd}
\opn\size{size}
\opn\div{div} \opn\Div{Div} \opn\cl{cl} \opn\Cl{Cl}
\opn\Spec{Spec} \opn\Supp{Supp} \opn\supp{supp} \opn\Sing{Sing}
\opn\Ass{Ass} \opn\Min{Min}
\opn\Ann{Ann} \opn\Rad{Rad} \opn\Soc{Soc}
\opn\Im{Im} \opn\Ker{Ker} \opn\Coker{Coker} \opn\Am{Am}
\opn\Hom{Hom} \opn\Tor{Tor} \opn\Ext{Ext} \opn\End{End}
\opn\Aut{Aut} \opn\id{id}  \opn\deg{deg}

\opn\nat{nat}
\opn\pff{pf}
\opn\Pf{Pf} \opn\GL{GL} \opn\SL{SL} \opn\mod{mod} \opn\ord{ord}
\opn\Gin{Gin} \opn\Hilb{Hilb}
\opn\aff{aff} \opn\con{conv} \opn\relint{relint} \opn\st{st}
\opn\lk{lk} \opn\cn{cn} \opn\core{core} \opn\vol{vol}
\opn\link{link} \opn\star{star}
\opn\gr{gr}

\def\pot#1#2{#1[\kern-0.28ex[#2]\kern-0.28ex]}

\opn\dirlim{\underrightarrow{\lim}}
\opn\inivlim{\underleftarrow{\lim}}
\def\Implies{\ifmmode\Longrightarrow \else
        \unskip${}\Longrightarrow{}$\ignorespaces\fi}
\def\implies{\ifmmode\Rightarrow \else
        \unskip${}\Rightarrow{}$\ignorespaces\fi}
\def\iff{\ifmmode\Longleftrightarrow \else
        \unskip${}\Longleftrightarrow{}$\ignorespaces\fi}

\let\:=\colon
\newtheorem{Theorem}{Theorem}[section]
\newtheorem{Lemma}[Theorem]{Lemma}
\newtheorem{Corollary}[Theorem]{Corollary}
\newtheorem{Proposition}[Theorem]{Proposition}
\newtheorem{Remark}[Theorem]{Remark}

\newtheorem{Example}[Theorem]{Example}

\let\epsilon\varepsilon
\let\phi=\varphi
\let\kappa=\varkappa
\def\qed{\ifhmode\textqed\fi
      \ifmmode\ifinner\quad\qedsymbol\else\dispqed\fi\fi}
\def\textqed{\unskip\nobreak\penalty50
       \hskip2em\hbox{}\nobreak\hfil\qedsymbol
       \parfillskip=0pt \finalhyphendemerits=0}
\def\dispqed{\rlap{\qquad\qedsymbol}}

\opn\dis{dis}
\def\pnt{{\raise0.5mm\hbox{\large\bf.}}}

\opn\Lex{Lex}

\begin{document}

\title{\bf Stanley depth of multigraded modules}

\author{ Dorin Popescu}

\thanks{The author was supported by CNCSIS Grant ID-PCE no. 51/2007. }

\address{Dorin Popescu, Institute of Mathematics "Simion Stoilow",
University of Bucharest, P.O.Box 1-764, Bucharest 014700, Romania}
\email{dorin.popescu@imar.ro} \maketitle

\begin{abstract}
 The Stanley's Conjecture on Cohen-Macaulay  multigraded modules is studied
 especially in dimension 2. In codimension 2 similar  results were obtained by  Herzog,
 Soleyman-Jahan and  Yassemi. As a consequence of our results
 Stanley's
 Conjecture holds in 5 variables.

  \vskip 0.4 true cm
 \noindent
\end{abstract}

\section*{Introduction}

Let $K$ be a field, $S=K[x_1,\ldots,x_n]$ be the polynomial ring in
$n$ variables, and $I\subset S$ a monomial ideal. One famous
conjecture of Stanley \cite{St} asked if there exist a presentation
of $S/I$ as a direct sum of $K$ linear spaces of the form $uK[Z]$
with $u$ monomial and $Z\subset \{x_1,\ldots,x_n\}$ a subset with
$|Z|\geq \depth S/I$. If this happens for a certain $I$ we say that
$I$ is a {\em Stanley ideal}. This conjecture is proved for $n\leq
4$ (see \cite{A1}, \cite{So}, \cite{AP1}). Here we show that this is
also true for $n=5$ (see Theorem \ref{5}). Also we notice that for
$n=5$ the so called pretty clean ideals \cite{HP} are exactly the
sequentially Cohen-Macaulay ones (see Theorem \ref{seq}). A similar
result for $n=4$ is given in \cite{AP}.

The proof of Theorem \ref{5} forces us to study  Stanley's
Conjecture on multigraded $S$-modules inspired by \cite{HVZ}. We
prove that some Cohen-Macaulay multigraded $S$-modules of dimension
$2$ are clean (see Theorem \ref{main}) in Dress terminology
\cite{Dr}, and so  Stanley's Conjecture holds for them. The proof
uses a partial polarization (see Proposition \ref{depth}) and is
hard and long   starting  with Lemma \ref{prime}.  This lemma is a
particular case of Lemma \ref{clean4} and their proofs are similar.
We believe that Lemma \ref{prime} deserves to be separately
considered, because it applies directly to a nice result (Theorem
\ref{clean3}). The value of Lemma \ref{clean4} appears later in the
next section, namely in Theorem \ref{pclean}. Several examples show
why the methods should be special for each case. A particular case
of Theorem \ref{main} says that monomial Cohen-Macaulay ideals of
dimension $2$ are Stanley ideals (see Corollary \ref{clean6}). This
reminds us \cite[Proposition 1.4]{HSY}, which says that monomial
Cohen-Macaulay ideals of codimension $2$ are Stanley ideals.  Some
Cohen-Macaulay multigraded  modules $M$ having only an associated
prime ideal are clean independently of dimension, as shows Theorem
\ref{maincor}.

We owe thanks to J. Herzog for some useful comments on our Lemma
\ref{prime} and the idea to state most of our results for the
filtration depth instead Stanley depth as they appeared in an
earlier version of our paper.

\section{Clean Multigraded Modules}

Let $K$ be a field, $S=K[x_1,\ldots,x_n]$, $n\geq 2$ be the
polynomial ring in $n$ variables, $m=(x_1,\ldots, x_n)$ and $M$ be a
finitely generated $\ZZ^n$-graded (i.e. multigraded) $S$-module. Let
$${\mathcal F}:\ \ 0=M_0\subset M_1\subset\ldots\subset M_r=M$$
be a chain of multigraded submodules of $M$. Then ${\mathcal F}$ is
a {\em prime filtration} if $M_i/M_{i-1}\cong S/P_i(-a_i)$, where
$a_i\in {\NN^n}$ and $P_i$ is a monomial prime ideal, $i\in [r]$.
The set $\Supp({\mathcal F})=\{P_1,\ldots, P_r\}$ is called the {\em
support} of ${\mathcal F}$. The filtration ${\mathcal F}$ is called
{\em clean} if $\Supp({\mathcal F})=\Min(M)$. If $\Ann_S (M)$ is
reduced then $M$ is called {\em reduced}. When $M$ is a reduced
cyclic $S$-module then $M$ is clean if and only if the simplicial
complex associated to $M$ is shellable (non-pure) as Dress shows
\cite{Dr}.

Usually, we will restrict to study multigraded $S$-modules $M$ with
$\dim_KM_a\leq 1$ for all $a\in {\ZZ}^n$. If $M$ is such a module
and $U\subset M$ is a multigraded submodule then an element $x=
\sum_{a\in {\ZZ}^n} x_a$ of $M$ belongs to $U$ if and only if all
${\ZZ}^n$-homogeneous components $x_a$ of $x$ belongs to $U$. Let
$I\subset J$ be monomial ideals of $S$. Then $M=J/I$ satisfies
$\dim_KM_a\leq 1$ for all $a\in {\ZZ}^n$.

  The following elementary lemma is known in a more
general frame (see \cite[Corollary 2.2]{Po_1}). We will prove it
here for the sake of our completeness.

\begin{Lemma} \label{clean} Let  $M$ be a  multigraded $S$-module
with $\Ass M=\{P_1,\ldots,P_r\}$,

\noindent $\dim S/P_i=1$ for $i\in [r]$.  Let $0=\cap_{i=1}^{r}N_i$
be an irredundant primary decomposition of $(0)$ in $M$ and suppose
that  $P_i=\Ann(M/N_i)$ for all $i$. Then $M$ is clean.
\end{Lemma}
\begin{proof}
 Apply induction on $r$. If $r=1$ then $M$ is torsion-free over
$S/P_1$ and  we get $M$ free over $S/P_1$ since $\dim(S/P_1)= 1$.
Thus $M$ is clean over $S$.

Suppose that $r>1$. Then $0=\cap_{i=1}^{r-1}(N_i\cap N_r)$ is an
irredundant primary decomposition of $(0)$ in $N_r$ and by induction
hypothesis we get $N_r$ clean. Also $M/N_r$ is clean because
$\Ass(M/N_r)=\{P_r\}$ (case $r=1$). Hence the filtration $0\subset
N_r\subset M$ can be refined to a clean filtration of $M$.
\end{proof}

Next we will extend the above lemma for some reduced Cohen-Macaulay
multigraded modules of dimension 2 (from now on we suppose that
$n>2$). But first we need some preparations.

\begin{Lemma}\label{prime} Let $I\subset U$ be two monomial ideals
of $S$ such that $\Ass\; S/I$ contains only prime ideals of
dimension $2$, and $\Ass\; S/U$ contains only prime ideals of
dimension $1$. Suppose that $I$ is reduced. Then there exists $p\in
\Ass\; S/I$ such that $U/(p\cap U)$ is a Cohen-Macaulay module of
dimension $2$.
\end{Lemma}
\begin{proof}Let $U=\cap_{j=1}^t Q_j$ be a reduced primary
decomposition of $U$ and $P_j=\sqrt{Q_j}$. Since $P_1\supset I$
there exists a minimal prime ideal $p_1$ containing $I$ such that
$P_1\supset p_1$. Thus $P_1=(p_1,x_k)$ for some $k\in [n].$ We may
suppose that $p_1=(x_2,\ldots, x_{n-1})$ and $k=n$   after
renumbering the variables. Using the description of monomial primary
ideals we note that $Q_1+p_1=(p_1,x_n^{s_1})$ for some positive
integer $s_1$. Let
$${\mathcal I}=\{i\in [t]:P_i=(p_i,x_n)\ \ \text{for \ \ some}\ \
p_i\in \Min\; S/I\}.$$
 Clearly, $1\in {\mathcal I}$. If $i\in
{\mathcal I}$, that is $P_i=(p_i,x_n)$, then as above
$Q_i+p_i=(p_i,x_n^{s_i})$ for some positive integers $s_i$. We may
suppose that $s_1=\max_{i\in {\mathcal I}} s_i$. Then we claim that
$\depth \; S/(p_1+U)=1$.

Let $r\in [t]$ be such that $p_1+Q_r$ is $m$-primary. Since
$$U=U+I=\cap_{j\in [t], q\in \Ass\; S/I}(q+Q_j)$$ and $\depth\;
S/U=1$ we see that  $p_1+Q_r$ contains an intersection $\cap_{j=1}^e
(q_j+Q_{c_j})$ with $q_j\in \Ass\ S/I$, $c_j\in [t]$ and $\dim
(q_j+Q_{c_j})=1$, that is $q_j\subset P_{c_j}$. Note that $p_1+Q_r=$
$$p_1+(x_1^{a_1},x_n^{a_n},\ \ \text{some}\ x^{\alpha}\ \text{with}\ \supp\;
\alpha=\{1,n\}).$$

Suppose that $x_n\in P_{c_j}$ for all $j\in [e]$. Then we show that
$p_1+Q_1\subset p_1+Q_r$. Indeed, by hypothesis a power of $x_n$
belongs to the minimal system of generators of $q_j+Q_{c_j}$, let us
say $x_n^{b_{c_j}}\in G(q_j+Q_{c_j}).$ Set $b=\max_j b_{c_j}$. Then
$x_n^b\in \cap_{j=1}^e (q_j+Q_{c_j})\subset p_1+Q_r$ and so $b\geq
a_n$. Let $b=b_{c_j}$ for some $j$. If $x_n\in q_j$ then $b=1$ and
so $a_n=1$ and clearly $p_1+Q_1\subset P_1\subset p_1+Q_r$. If
$x_n\not \in q_j$ then $c_j\in {\mathcal I}$ and so $b=s_{c_j}\leq
s_1$. Thus $p_1+Q_1=(p_1,x_n^{s_1})\subset (p_1,x_n^b)\subset
p_1+Q_r$.

Now suppose that there exists $j\in [e]$ such that $x_n\not \in
P_{c_j}$, let us say $j=1$. Then $P_{c_1}=(x_1,\ldots, x_{n-1})$ and
$x_n\in P_{c_j}$ for all $j>1$. As above, let $x_n^{b_{c_j}}\in
G(q_j+Q_{c_j})$ for $j>1$ and $b=\max_{j>1}^e b_{c_j}$. If $b\geq
a_n$ then as above $p_1+Q_1\subset p_1+Q_r$. Assume that $b<a_n$.
Let $x_1^d$ be the power of $x_1$ contained in $G(q_1+Q_{c_1})$. If
$d\geq a_n$ then we get $p_1+Q_{c_1}=(p_1,x_1^d)\subset p_1+Q_r$ and
$\dim (p_1+Q_{c_1})=1$. Suppose that $d<a_1$. Then note that
$x_1^dx_n^b\in G(\cap_{j=1}^e (q_j+Q_{c_j}))$ and so $x_1^dx_n^b\in
p_1+Q_r$.  Thus $$(p_1+Q_1)\cap
(p_1+Q_{c_1})=(p_1,x_1^dx_n^b)\subset p_1+Q_r.$$

 Hence $p_1+U$ is the intersection
of primary ideals of dimension $ 1$, that is

\noindent $\depth S/(p_1+U)=1$. From the exact sequence
$$0\rightarrow U/(p_1\cap U)\rightarrow S/p_1\rightarrow S/(p_1+U)\rightarrow 0$$
we get $\depth U/(p_1\cap U)=2$.
\end{proof}

\begin{Lemma} \label{clean2} Let  $I\subset U$ be two monomial ideals
such that $U/I$ is a Cohen-Macaulay $S$-module of dimension $2$ and
$\Ass S/I$ contains only prime ideals of dimension $2$. Suppose that
$I$ is reduced. Then there exists $p\in \Ass S/I$ such that
$U/(p\cap U)$ is a Cohen-Macaulay module of dimension $2$.
\end{Lemma}

\begin{proof} Let $U=\cap_{j=1}^t Q_j$ be a reduced primary
decomposition of $U$ and $P_j=\sqrt{Q_j}$. From the exact sequence
$$0\rightarrow U/I\rightarrow S/I\rightarrow S/U\rightarrow 0$$
we get $$\depth S/U\geq  \min\{\depth U/I -1, \depth S/I\}\geq 1.$$
Thus $1\leq \dim S/P_j\leq 2$ for all $j$. Suppose that $\dim
S/P_j=2$. Then we have $\Pi_{q\in \Ass S/I}\ q\subset I\subset
U\subset Q_j\subset P_j$. Thus $P_j\supset p $ for some $p\in \Ass
S/I$ and we get $P_j=p$ because $\dim S/P_j=\dim S/p$.  It  follows
that $\Pi_{q\in \Ass S/I,q\not =p }q\not \subset P_j$ and so
$p\subset Q_j\subset P_j=p$. Hence $p=Q_j$. Set $U'=\cap_{i=1, i\not
=j}^t Q_i$, $I'=\cap_{q\in \Ass S/I, q\not =p} \ q$. Then
$$(U+I')/I'\cong U/(U\cap I')=U/(U'\cap p\cap I')=U/(U'\cap I)=U/I.$$
 Changing $I$ by $I'$ and $U$ by $U+I'$ we may
reduce to a smaller $|\Ass S/I|$. By recurrence we may reduce in
this way to the case when $\dim S/Q_j=1$ for all $j\in [t]$ since
$I\not =U$. Now is enough to apply the above lemma.
\end{proof}

The above lemma cannot be extended to show that $U/(p\cap U)$ is
Cohen-Macaulay for all $p\in \Ass S/I$, as shows the following:

\begin{Example} {\em Let $$I=(x_1x_2)\subset U=(x_1,x_3^2)\cap
(x_2,x_3)$$ be monomial ideals of $S=K[x_1,x_2,x_3]$. We have
$\depth S/U=1$ and $\depth S/I=2$. Thus $U/I$ is Cohen-Macaulay but
$U/(U\cap (x_2))$ is not since $U+(x_2)=(x_1,x_2,x_3^2)\cap
(x_2,x_3)$, that is $\depth S/(U+(x_2))=0$. However, $U/(U\cap
(x_1))$ is Cohen-Macaulay because $U+(x_1)=(x_1,x_3^2)$.}
\end{Example}

\begin{Remark} {\em A natural extension of Lemma \ref{prime} is to
consider the case when $I$ is not reduced, let $I=\cap_{i=1}^r q_i$
be a reduced primary decomposition, and to show that $U/U\cap q_j$
is a Cohen-Macaulay module of dimension 2 for a certain $j$.
Unfortunately, this is not true as shows the following:}
\end{Remark}

\begin{Example} {\em Let $U=(x_1,x_3^2)\cap (x_2^2,x_3)$ and
$I=(x_1^2x_2^2)$ be monomial ideals of $S=K[x_1,x_2,x_3]$. Since
$\depth S/U=1$, $\depth S/I=2$ we see that $U/I$ is a Cohen-Macaulay
module of dimension 2. But $U+(x_1^2)=(x_1,x_3^2)\cap
(x_1^2,x_2^2,x_3)$ and $U+(x_2^2)=(x_1,x_2^2,x_3^2)\cap (x_2^2,x_3)$
show that $\depth S/(U+(x_1^2))=\depth S/(U+(x_2^2))=0$ and so
$U/(U\cap(x_1^2))$ and $U/(U\cap (x_2^2))$ are not Cohen-Macaulay.
However, $U+(x_1)=(x_1,x_3^2)$ and so $U/(U\cap (x_1))$ is a
Cohen-Macaulay module of dimension 2.}
\end{Example}

The idea of the following theorem comes from the fact that connected
simplicial complexes of dimension 1 are shellable (see \cite{Po_2}
for details).

\begin{Theorem} \label{clean3} Let  $I\subset U$ be two monomial ideals
such that $U/I$ is a Cohen-Macaulay $S$-module of dimension $2$.
Suppose that $I$ is reduced. Then $U/I$ is clean.
\end{Theorem}

\begin{proof} We follow  the proof of Lemma
\ref{clean}. Let $I=\cap_{i=1}^r p_i$ be a reduced primary
decomposition of $I$ (so $p_i$ are prime ideals). Apply induction on
$r$. If $r=1$ then $U/I$ is a maximal Cohen-Macaulay (so free) over
$S/p_1$. Thus $U/I$ is clean. Suppose that $r>1$. Then there exists
$j\in [r]$ such that $U/(p_j\cap U)$ is  a Cohen-Macaulay module of
dimension 2 using Lemma \ref{clean2}. From the exact sequence
$$0\rightarrow (p_j\cap U)/I\rightarrow U/I\rightarrow U/(p_j\cap U)\rightarrow 0$$
we see that $(p_j\cap U)/I$ is a  Cohen-Macaulay module of dimension
2. Set $I'=\cap_{i=1,i\not =j}^r p_i$. We have $(p_j\cap U)/I\cong
((p_j\cap U)+I')/I'$ because $(p_j\cap U)\cap I'=U\cap I=I$.
Applying induction hypothesis we get  the modules $((p_j\cap
U)+I')/I'$ and  $(U+p_j)/p_j\cong U/(p_j\cap U)$ clean and so the
filtration $0\subset (p_j\cap U)/I\subset U/I$ can be refined to a
clean one.
\end{proof}

\begin{Remark} {\em The above theorem does not hold for dimension 3.
The triangulation of the real projective plane gives a non-shellable
simplicial complex of dimension 2, which is Cohen-Macaulay over a
field $K$ of characteristic $\not =2$. So its associated
Stanley-Reisner ring is Cohen-Macaulay but not clean (see
\cite{Po_2} for details). Thus Lemma \ref{prime} cannot be extended
for depth 3, because otherwise the proof of the above proposition is
valid also in this case. An idea of this counterexample is given in
the following:}
\end{Remark}

\begin{Example} {\em Let $K$ be a field of characteristic  $\not =2$,
$S=K[a,b,c,d,e,f]$, $J=(a,c,f)\cap (b,e,f)\cap (c,d,e)\cap (c,e,f)$
and $$I=J\cap (a,b,c)\cap (a,b,e)\cap (a,d,e)\cap (a,d,f)\cap
(b,c,d)\cap (b,d,f).$$ Set $M=J/I$. We have the following exact
sequence
$$0\rightarrow M\rightarrow S/I \rightarrow S/J \rightarrow 0.$$
$S/J$ is shellable because the associated simplicial complex has the
facets $\{abf\}$,

\noindent $\{abd\},\{adc\},\{bde\}$ written in the shelling order.
$S/I$ is Cohen-Macaulay since its associated simplicial complex is
the triangulation of the real projective plane (see \cite[page
236]{BH} for details). Thus $\depth(S/I)=\depth(S/J)=3$. Applying
Depth Lemma to the above exact sequence we get $\depth\ M\geq 3$
(actually equality since $\dim\ M=3$). Set $P=(a,b,c)$. Then $P/I$
is the $P$-primary component of $S/I$ and so $N=P\cap J/I$ is the
$P$-primary component of $M$. Thus $M/N\cong (J+P)/P$. But
$J+P=(a,b,c,d,e)\cap (a,b,c,f)$ and so $\depth(S/(J+P))=1$. From the
exact sequence
$$0\rightarrow M/N \rightarrow S/P \rightarrow S/(J+P) \rightarrow 0$$
we get $\depth(M/N)=2<\depth\ M$ (a maximal regular sequence on
$M/N$ is $\{d,e-f\}$).}
\end{Example}

Next we want to extend Lemma \ref{prime} in the case when $I$ is not
reduced. First we present some preliminaries. Let $I\subset S$ be a
monomial ideal and $G(I)$ its unique minimal system of monomial
generators. Let $k\in [n]$. If $x_k$ is not regular on $S/I$ then
let $d_k(I)$ be the highest degree of $x_k$, which appears in a
monomial of $ G(I)$. If $x_k$ is regular on $S/I$, that is $x_k$
does not appear in a monomial of $G(I)$ then set $d_k(I)=0$. Let
$d(I)=\max_{k\in [n]} d_k(I)$.

\begin{Example}\label{d}{\em Let $Q$ be a monomial $(x_1,\ldots,x_r)$-primary ideal,
$r\leq n$. Then $Q=
(x_1^{a_1},\ldots,x_r^{a_r},\{x^{\alpha}:\alpha\in \Gamma,
\supp(\alpha)\subset\{1,\ldots,r\},|\supp(\alpha)|>1\})$ by
\cite[Proposition 5.1.8]{Vi} for  a set $\Gamma\subset\NN^n$, some
positive integers $a_1,\ldots, a_r$ and $d_k(Q)=a_k$ for $k\in [r]$,
$d(Q)=\max_{k\in [r]}a_k$. If $U=\cap_{j=1}^t Q_j$ is a reduced
primary decomposition of the monomial ideal $U$ then
$d_k(U)=\max_{j\in [t]}d_k(Q_j)$ and $d(U)=\max_{j\in [t]}d(Q_j)$.}
\end{Example}

\begin{Lemma} \label{clean4} Let  $I\subset U$ be two monomial ideals
of $S$ such that  $\Ass S/I$ contains only prime ideals of dimension
$2$, and $\Ass S/U$ contains only prime ideals of dimension $1$. Let
$k\in [n]$ and  $U=\cap_{i=1}^t Q_i$,  be a reduced primary
decomposition of $U$,  $P_i=\sqrt{Q_i}$. Suppose that
$d(I)=d_k(I)\leq d_k(Q_i)$ for some $i$ and there exists $p\in \Ass
S/I$ such that $P_i=(x_k,p)$. Then  $U/(p\cap U)$ is a
Cohen-Macaulay module of dimension $2$.
\end{Lemma}

\begin{proof} Let  $I=\cap_{j=1}^e q_j$ be a reduced primary
decomposition of  $I$ and  $p_j=\sqrt{q_j}$.
 After
renumbering the variables $x$ and ideals $(Q_i)_{i}, (q_j)_j$ we may
suppose $k=n$, $i=1$, $p=p_1=(x_2,\ldots, x_{n-1})$, $P_1=(p_1,x_n)$
and $d_n(Q_1)\geq d_n(I)$.  Let
$${\mathcal I}=\{i\in [t]:\text{there\ \ is}\ j\in [e]\ \text{with} \ \ P_i=(p_j,x_n)\}.$$
 Clearly, $1\in {\mathcal I}$. Renumbering $(P_i)$, $(p_j)$ we may
 suppose that $d_n(Q_1)=\max_{i\in {\mathcal I}} d_n(Q_i)$. Note
 that we have still $d_n(Q_1)\geq d_n(I)$.
  Then we
 claim that $ \depth S/(p_1+U)=1$.

Let $r\in [t]$ be such that $p_1+Q_r$ is $m$-primary. Since
$$U=U+I=\cap_{i\in [t], j\in [e]}(q_j+Q_i)$$ and $\depth\;
S/U=1$ we see that  $q_1+Q_r$ contains an intersection $\cap_{j\in
{\mathcal J}} (q_j+Q_{c_j})$ for a subset ${\mathcal J}\subset [e]$
and $c_j\in [t]$ with $\dim (q_j+Q_{c_j})=1$, that is $p_j\subset
P_{c_j}$. Note that $$p_1+Q_r= p_1+(x_1^{a_1},x_n^{a_n},\ \
\text{some}\ x^{\alpha}\ \text{with}\ \supp\; {\alpha}=\{1,n\}).$$

Suppose that $x_n\in P_{c_j}$ for all $j\in {\mathcal J}$. Then we
show that $p_1+Q_1\subset p_1+Q_r$. Indeed, we set $b=\max_{j\in
{\mathcal J}} d_n(q_j+Q_{c_j})$ and note that  $x_n^b\in \cap_{j\in
{\mathcal J}} (q_j+Q_{c_j})\subset q_1+Q_r\subset p_1+Q_r$ and so
$b\geq a_n$. Let $b=d_n(q_{\rho}+Q_{c_{\rho}})$ for some $\rho\in
{\mathcal J}$. If $x_n\in p_{\rho}$ then we get $d_n(q_{\rho})\geq
d_n(q_{\rho}+Q_{c_{\rho}})=b\geq a_n$. By hypothesis, we have
$d_n(Q_1)\geq d_n(I)\geq d_n(q_{\rho})\geq b\geq a_n$ and so
$p_1+Q_1\subset p_1+Q_r$. If $x_n\not \in p_{\rho}$ then
$c_{\rho}\in {\mathcal I}$ and we have
$$d_n(Q_1)\geq d_n(Q_{c_{\rho}})\geq d_n(q_{\rho}+Q_{c_{\rho}})=b\geq
a_n$$ and so we get again  $p_1+Q_1\subset p_1+Q_r$.

Now suppose that there exists $\nu \in {\mathcal J}$ such that
$x_n\not \in P_{c_{\nu}}$. Then $P_{c_{\nu}}=(x_1,\ldots, x_{n-1})$
and $x_n\in P_{c_j}$ for all $j\in {\mathcal J}$ with $j\not =\nu$.
Set $b=\max_{j\in {\mathcal J},j\not =\nu }d_n(q_j+Q_{c_j})$. If
$b\geq a_n$ then as above $p_1+Q_1\subset p_1+Q_r$. Assume that
$b<a_n$. If $d=d_1( Q_{c_{\nu}})\geq a_1$  then we get
$p_1+Q_{c_{\nu}}\subset p_1+Q_r$ and $\dim (p_1+Q_{c_{\nu}})=1$.
Suppose that $d<a_1$. Then note that $x_1^dx_n^b\in \cap_{j\in
{\mathcal J}} (q_j+Q_{c_j}))\subset p_1+Q_r$.  Thus
$$(p_1+Q_1)\cap (p_1+Q_{c_{\nu}})=(p_1,x_1^dx_n^b)\subset p_1+Q_r.$$

 Hence $p_1+U$ is the intersection
of primary ideals of dimension $ 1$, that is

\noindent $\depth S/(p_1+U)=1$. From the exact sequence
$$0\rightarrow U/(p_1\cap U)\rightarrow S/p_1\rightarrow S/(p_1+U)\rightarrow 0$$
we get $\depth U/(p_1\cap U)=2$.
\end{proof}

\section{Partial polarization}

The aim of this section is to extend slightly a part of
\cite[Theorem 2.6]{HTT} (such tool is needed for a generalization of
Lemma \ref{clean4}, see Theorem \ref{pclean}). Let $I\subset S$ be a
monomial ideal. We have $d(I)=1$ if and only if $I$ is square free.
Suppose that $d(I)>1$. We consider the monomial ideal $I_1$
generated by $I$ and the monomials $u/x_i$ for those $i\in [n]$ and
$u\in G(I) $ with $\deg_{x_i}u=d(I)$. Clearly $d(I_1)=d(I)-1$ and
$I_1/I$ is reduced. Let $U\supset I$ be a monomial ideal. Similarly,
let $U_1$ be the ideal generated by $U$ and the monomials $a/x_i$
for those $i\in [n]$ and $a\in G(U) $ with $\deg_{x_i}a=d(I)$.

\begin{Proposition}\label{depth}
 Then $U_1\supset I_1$ and $$\depth(U_1/I_1)\geq \depth(U/ I)$$ if $U_1\not =I_1$. In
particular  $\depth_S S/I_1\geq \depth_S S/I.$
\end{Proposition}

The proof follows applying recursively the following lemma for
different $i \in [n]$. Fix $i \in [n]$ such that there exists $v\in
G(I)$ with $\deg_{x_i}v=d(I)$ and let $\tilde I$ be the ideal
generated in $S$ by $I$ and the monomials $v/x_i$ with $v\in G(I)$,
$\deg_{x_i}v=d(I)$. Let $U\supset I$ be a monomial ideal. Similarly,
let ${\tilde U}$ be the ideal generated by $U$ and the monomials
$a/x_i$  with $a\in G(U) $, $\deg_{x_i}a=d(I)$.

\begin{Lemma}\label{depth1}
 Then  ${\tilde U}\supset {\tilde I}$ and
$$\depth({\tilde U}/{\tilde I})\geq \depth(U/I)$$ if ${\tilde U}\not ={\tilde I}$.
In particular
$\depth_S S/{\tilde I}\geq \depth_S S/I.$
 \end{Lemma}

\begin{proof} We follow the proof of \cite[Theorem 2.6]{HTT} and of
\cite[Lemma 4.2.16]{BH} slightly modified.  We introduce a new
variable $y$ and   set
 $u'=x_1^{a_1}\cdots
x_{i-1}^{a_{i-1}}x_i^{a_i-1}yx_{i+1}^{a_{i+1}}\cdots x_n^{a_n}$ for
$u=x_1^{a_1}\cdots x_n^{a_n}\in G(I)$ if $a_i=d(I)$ and $u'=u$
otherwise. Let $I'$ be the ideal generated in $T=S[y]$ by $\{u':u\in
G(I)\}$. We claim that $y-x_i$ is regular on $T/I'$. Suppose this is
not the case. Then $y-x_i$ belongs to an associated prime ideal $Q$
of $T/I'$. Since $Q$ is a monomial ideal we get $y,x_i\in Q$ and we
may choose a monomial $f$ such that $Q=(I':f)$. Thus $yf,x_if\in
I'$. If $yf=wu'$ for some $u'\in G(I')$ and a monomial $w$ we get
$y|u'$ because if $y|w$ then $f\in I'$, which is false. Then
$x_i^{d(I)-1}|u'$ and so $x_i^{d(I)-1}|f$.

Now $x_if=qv' $ for some $v'\in G(I')$ and $q $ a monomial. Clearly
$q$ is not a multiple of $x_i$ because otherwise $f\in I'$. Thus
$x_i^{d(I)}|v'$ which is not possible since $\deg_{x_i}v'< d(I)$ by
construction. Hence $y-x_i$ is regular on $T/I'$.

As above, let  $U'$ be the ideal generated in $T$ by $a'=ya/x_i$ for
$a\in G(U)$ with $\deg_{x_i} a=d(I)$ and by $\{a\in G(U):\deg_{x_i}
a\not=d(I)\}$. Note that $U'\supset I'$. Set $M'=U'/I'$ and $M=U/I$.
We have
$$M'/(y-x_i)M'\cong U'/(I'+(y-x_i)U')$$ and consider the surjective
map $\varphi: M'/(y-x_i)M'\rightarrow M$ given by $h\rightarrow
h(y=x_i)$. For the injectivity of $\varphi $ it is enough to show
the inclusion
$$[U'\cap(y-x_i)T]+I'\subset I'+(y-x_i)U'.$$
Let $f\in T$ be such that $(y-x_i)f\in U'$. We show that $f\in U'$.
As $U'$ is a monomial ideal it is enough to take $f$ monomial. Then
$yf,x_if\in U'$. If $yf\in (U'\cap S)T$ then $f\in (U'\cap
S)T\subset U'$. Otherwise, $yf=a'g$ for some $g\in T$ and $a\in
G(U)$ with $\deg_{x_i}(a)=d(I)$. If $y|g$ then $f=a'(g/y)\in U'$. If
$y\not |g$ then $f=(u/x_i)g\in S$. As $x_if\in U'$ it follows by
construction that  $x_if=wq$ for some $q\in S$ and  $w\in G(U)$ with
$\deg_{x_i}w\not =d(I)$. If $x_i|q$  then $f\in U'$. Otherwise,  we
get $\deg_{x_i}w=d(I)$, which is false. Hence $\varphi$ is an
isomorphism. The module $M'$ could be seen as the {\em first step of
polarization of} $M$ {\em with respect of $x_i$}.

Next, let $N$ be the multiplicative system generated by $y$ in $T$.
We have ${\tilde I}T_N=I'T_N$ and ${\tilde U}T_N=U'T_N$. Tensorizing
with $S\otimes_T-$ a minimal free resolution of $M'$ over $T$ we get
a minimal free resolution of $M$ over $S$ because $y-x_i$ is regular
on $T/I'$ and so on $M'$. Thus $\pd_T M'=\pd_S M.$ Since the
localization is exact it follows
$$\pd_{T_N} {\tilde U}T_N/{\tilde I}T_N=\pd_{T_N} M'_N\leq \pd_TM'=\pd_SM.$$ But
$$\pd_{T_N} {\tilde U}T_N/{\tilde I}T_N=\pd_{S} {\tilde U}/{\tilde
I},$$
  and so
$$\pd_S {\tilde U}/ {\tilde I}\leq \pd_SM = \pd_S U/I.$$ Applying
Auslander-Buchsbaum Theorem we are done.
\end{proof}

\begin{Remark}{\em The polarization can be made even with a variable
$x_i$ for which $d_i(I)<d(I)$ and the above lemma still holds.}
\end{Remark}

The following examples show some limits of our Lemma \ref{depth1}.

\begin{Example} {\em Let $I=(x^2,z^2)$, $U=(xy)+I$ be ideals in
$S=K[x,y,z]$. Fix the variable $x$.
 We have ${\tilde I}=(x,z^2)={\tilde U}$, but
$U/I\cong S/((x^2,z^2):xy)=S/(x,z^2)$ has depth 1.}
\end{Example}

\begin{Example}\label{mainex} {\em Let $I=(x^2z,t^2)$, $U=(x^2,xy,t^2)$ be ideals in
$S=K[x,y,z,t]$. Fix the variable $x$. We have ${\tilde I}=(xz,t^2)$,
${\tilde U}=(x,t^2)$. Since $y$ is regular on $S/I$, it is also
regular on $U/I$ and we have $U/(I+yU)=U/((x,t^2)\cap
(x^2,y^2,t^2)\cap (z,y,t^2))$. Thus $z+x$ is regular on $U/(I+yU)$
and so $U/I$ has depth 2. Clearly, ${\tilde U}/{\tilde I}\cong
S/(z,t^2)$ has depth 2 as supports also Lemma \ref{depth1}. In the
following exact sequence
$$0\rightarrow  (U+{\tilde I})/{\tilde I}\rightarrow
S/(xz,t^2)\rightarrow S/(x^2,xy,xz,t^2)\rightarrow 0,$$ we have
$\depth(S/(xz,t^2))=2$ and $\depth(S/(x^2,xy,xz,t^2))=0$,
$(x^2,xy,xz,t^2)=(x,t^2)\cap (x^2,y,z,t^2)$ being a reduced primary
decomposition. It follows
$$1=\depth(U/(U\cap {\tilde I}))< \depth(U/I)=\depth ({\tilde U}/{\tilde
I})=2.$$}
\end{Example}

Next we give some sufficient conditions when the following
inequality holds: $\depth(U/(U\cap {\tilde I}))\geq \depth(U/I)$.

\begin{Lemma}\label{hope} Let $U\supset I$ be some monomial ideals such that
$U/I$ is a Cohen-Macaulay $S$-module of dimension $s$. Suppose that
$I$ is a primary ideal. If $U\not \subset {\tilde I}$ then $U/(U\cap
{\tilde I})$ is a Cohen-Macaulay $S$-module of dimension $s$.
\end{Lemma}
\begin{proof} As in the Example \ref{d}, we may suppose that

\noindent
$G(I)=\{x_1^{d_1},\ldots,x_r^{d_r},x^{u_1},\ldots,x^{u_p}\}$ for
some positive integers $d_j$, $r\in [n]$, and some $u_e\in {\NN}^n$
with support in $\{x_1,\ldots,x_r\}$ by renumbering the variables
$x$. Then $d_i=d(I)$ for some $i\in [r]$. Note that $G(U)$ contains
a power $x_i^t$ of $x_i$, $1\leq t\leq d(I)$ and there exist no
other monomial $a\in G(U)$ with $\deg_{x_i}a\geq t$. Hence ${\tilde
U}=U+{\tilde I}$. By Lemma \ref{depth1} we have
$$\depth(U/(U\cap {\tilde I}))=\depth({\tilde U}/{\tilde I})\geq
\depth(U/I).$$
\end{proof}
\begin{Corollary}\label{depth2} Let $U\supset I$ be some monomial ideals such that
$U/I$ is a Cohen-Macaulay $S$-module of dimension $s$ and $I$ is a
primary ideal.
 If $U\not \subset I_1$ then  $U/(U\cap I_1)$
is a Cohen-Macaulay $S$-module of dimension $s$.
\end{Corollary}
For the proof apply the above lemma recursively.

Next we want to complete Lemma \ref{clean4}.

\begin{Theorem} \label{pclean} Let  $I\subset U$ be two monomial ideals
of $S$ such that  $\Ass S/I$ contains only prime ideals of dimension
$2$, and $\Ass S/U$ contains only prime ideals of dimension $1$.
Then there exists $p\in \Ass S/I$ such that $U/(p\cap U)$ is a
Cohen-Macaulay module of dimension $2$.
\end{Theorem}

\begin{proof} Apply induction on $d(I)$. If $d(I)=1$ then apply Lemma \ref{prime}.
Suppose that $d(I)>1$.
 Let  $U=\cap_{j=1}^t Q_j$, $I=\cap_{s=1}^e q_s$ be some reduced primary decompositions of $U$, respectively $I$ ,
$P_j=\sqrt{Q_j}$ and $p_s=\sqrt{q_s}$. If there exists $k\in [n]$
such that $d(I)=d_k(I)\leq d_k(Q_j)$ for some $j$ and  $P_j=(x_k,p)$
for some $p\in \Ass S/I$, then we may apply Lemma \ref{clean4}.
Otherwise, for all $k\in [n]$, $j\in [t]$,  $s\in [e]$  with
$d(I)=d_k(I)\leq d_k(Q_j)$ and  $P_j\supset p_s $ we have $x_k\in
p_s$. Take $i$ such that $d(I)=d_i(I)$ and consider the construction
${\tilde I}$ with respect of this $i$. Set $E_{js}=Q_j+q_s$.
Clearly, ${\tilde E}_{js}\supset Q_j+{\tilde q_s}$. We claim that
${\tilde E}_{js}=Q_j+{\tilde q_s}$ if it is not $m$-primary. Indeed,
if $d(I)\leq d_i(Q_j)$ and $x_i^{d(I)}\in G(q_s)$ then
$x_i^{d(I)-1}\in {\tilde q}_s$. If $x_i^{d(I)}\not \in G(q_s)$ we
get $d(I)>d_i(q_s)$ and so $d_i(E_{js})<d(I)$ if $x_i\in p_s$. If
$x_i\not\in p_s$ then either $d(I)>d_i(Q_j)$ by our hypothesis if
$P_j\supset p_s$, or $E_{js}$ is $m$-primary. Hence, if
$x_i^{d(I)-1}\in {\tilde E}_{js}\setminus E_{js}$  and $E_{js}$ is
 not $m$-primary then we must have $x_i^{d(I)-1}\in {\tilde q}_s$, which
shows our claim.

As $U=U+I$ and $\depth S/U=1$ it follows $\tilde
U=\cap_j\cap_{s}{\tilde E}_{js}=\cap_j\cap_{s}(Q_j+{\tilde
q}_s)=U+{\tilde I}$ and ${\tilde U}\not ={\tilde I}$. By Lemma
\ref{depth1} we get that $U+{\tilde I}/{\tilde I}\cong{\tilde
U}/{\tilde I}$ is a Cohen-Macaulay module of dimension $2$. If
$I_1={\tilde I}$ then $d(I)>d({\tilde I})$ and we find by induction
hypothesis $p\in \Ass S/{\tilde I}=\Ass S/I$ such that $U/(p\cap
U)\cong {\tilde U}/(p\cap {\tilde U})$ is a Cohen-Macaulay module of
dimension $2$. If not, apply the same method by recurrence with
other variables $x_i$ arriving finally to $I_1$, where the induction
hypothesis work.
\end{proof}

\section{Stanley decompositions}

Let $M$ be a finitely generated multigraded $S$-module. Given a
$\ZZ^n$-homogeneous element $y\in M$ and $Z$ a set of variables
$\subset \{x_1,\ldots,x_n\}$ we set $yK[Z]=\{yf:f\in K[Z]\}$. The
$K$ linear space $yK[Z]$ is called a {\em Stanley space}, if $yK[Z]$
is a free  $K[Z]$-module. A presentation of $M$ as a direct sum of
Stanley spaces :
$${\mathcal D}:\ M= \oplus_{i=1}^r y_iK[Z_i]$$
is called a {\em Stanley decomposition} of $M$. Set
$\sdepth({\mathcal D})=\min_i |Z_i|$ and $$\sdepth(M)=\max_{\mathcal
D}\sdepth\ {\mathcal D},$$ where the maximum is taken over all
Stanley decompositions ${\mathcal D}$ of $M$. The {\em Stanley
depth} of $M$ is a nice invariant studied in \cite{R} and
\cite{HVZ}.

Let
$${\mathcal F}:\ \ 0=M_0\subset M_1\subset\ldots\subset M_r=M$$
be a {\em prime filtration} of $M$ as in the first section. After
\cite{HVZ} we define $\fdepth({\mathcal F})=\min\{\dim\ S/P:P\in
\Supp({\mathcal F})\}$ and the {\em filtration depth} of $M$ given
by $\fdepth(M)=\max_{\mathcal F}$\ $\fdepth({\mathcal F})$, where
the maximum is taken over all prime filtrations ${\mathcal F}$ of
$M$. We have
$$\fdepth(M)\leq \depth\ M, \sdepth\ M\leq \min\{\dim\ S/P:P\in
\Ass \ M\}$$ if $\dim_K\ M_a\leq 1$ for all $a\in \ZZ^n$ (see
\cite[Proposition 1.3]{HVZ}). If $M$ is clean then it has a clean
filtration ${\mathcal F}$, that is $\Supp({\mathcal F})=\Min\ M$,
and so it follows $\fdepth\ M=\depth \ M=\sdepth\ M$. Thus Lemma
\ref{clean} and Theorem \ref{clean3} give the following:

\begin{Proposition}\label{clean5}
Let $M$ be a multigraded $S$-module. Then $\fdepth\ M=\depth \
M=\sdepth\ M$ if one of the following conditions holds:
\begin{enumerate}
\item{} the associated prime ideals of $M$ have dimension $1$ and $\dim_K\ M_a\leq 1$ for all $a\in \ZZ^n$,
or
\item{} $M$ is a reduced Cohen-Macaulay module of dimension $2$ with
$M\cong U/I$ for some monomial ideals $I\subset U$ of $S$.
\end{enumerate}
\end{Proposition}
The above proposition has bellow two extensions.
\begin{Theorem}\label{maincor} Let $U\supset I$ be some monomial ideals such that
$U/I$ is a Cohen-Macaulay $S$-module  and $I$ is a primary ideal.
Then $U/I$ is clean. In particular, $\fdepth\ M=\depth\ M=\sdepth\
M$.
\end{Theorem}
\begin{proof}
Set $p=\sqrt{I}$ and apply induction on $d(I)$.  If $d(I)=1$ the
module $M=U/I$ is free (and so clean) over $S/p$. Suppose that
$d(I)>1$. If $U\subset I_1$ then $pU\subset I$ and so $M$ is again
free and clean over $S/p$. Otherwise, by Corollary \ref{depth2} we
see that $M_1=U/(U\cap I_1)$ is also  a Cohen-Macaulay module of
dimension $s=\dim \ M$. From the exact sequence
$$0\rightarrow N=(U\cap I_1)/(U\cap I)\rightarrow M\rightarrow M_1 \rightarrow 0$$
we get that $N$ is  a maximal Cohen-Macaulay module over $S/p$ of
dimension $s$, because $pN=0$. Thus  $N$ is free and clean as above.
Note that $M_1\cong (U+I_1)/I_1$ is clean by  the induction
hypothesis, because $d(I_1)<d(I)$. Hence the filtration $0\subset
N\subset M$ can be refined to a clean one.
\end{proof}

\begin{Theorem}\label{main}
Let $I\subset U$ be  some monomial ideals of $S$ such that $U/ I$ is
a Cohen-Macaulay $S$-module  of dimension $2$.  Then $U/I$ is clean.
In particular, $\fdepth(U/I)=\sdepth(U/I)=2.$
\end{Theorem}

\begin{proof} There exists a monomial ideal $J$ such that $I=U\cap
J$ and $\Ass(S/J)$ contains only prime ideals of dimension $2$.
Changing $I$ by $J$ and $U$ by $U+J$ we may assume that $\Ass S/I$
contains  only prime ideals of dimension $2$. Let $U=\cap_{i=1}^t
Q_i$, $I=\cap_{j=1}^e q_j$ be some reduced primary decompositions of
$U$ respectively $I$ and $P_i=\sqrt{Q_i}$, $p_j=\sqrt{q_j}$. Apply
induction on $e$. If $e=1$ apply Theorem \ref{maincor}.

Now suppose that $e>1$, $\dim S/P_1=2$ and there exists no $i\in
[t]$ such that $P_1\subset P_i$. It follows $P_1\supset p_j$ for
some $j\in [e]$, let us say $P_1\supset p_1$. Note that $P_1=p_1$,
$Q_1\supset q_1$.   Set $U'= \cap_{i>1}^t Q_i$, $I'=\cap_{j>1}^e
q_j$, $T=S\setminus ((\cup_{i>1}^t P_i)\cup (\cup_{j>1}^e p_j)$. We
have $I'\subset U'$. Note that $T^{-1}(U/I)=T^{-1}(U'/I')$ is
Cohen-Macaulay of dimension $2$, and so $U'/I'$ is Cohen-Macaulay of
dimension $2$ too. We show that $\depth S/(U+I')=1$. From the exact
sequence
$$0\rightarrow U /I\rightarrow Q_1/q_1\oplus U'/I'\rightarrow (Q_1+U')/(q_1+I') \rightarrow 0$$
we get $(Q_1+U')/(q_1+I')$ of depth $\geq 1$. Thus if $q_1+q_j$ is
$m$-primary then either $q_1+q_j\supset Q_1+U'$, or there exists
${\mathcal J}\subset [e]$ such that $\dim S/(q_1+q_k)= 1$ for all
$k\in {\mathcal J}$ and $\cap_{k\in {\mathcal J}}(q_1+q_k)\subset
q_1+q_j$.
 We have $U+I'=(Q_1+I')\cap U'=(\cap_{j=2}^t (Q_1+q_j))\cap U'$. If
$q_1+q_j\supset Q_1+U'$ then we get $Q_1+q_j\supset U+I'$. In the
second case we have  $\cap_{k\in {\mathcal J}}(q_1+q_k)\subset
q_1+q_j\subset Q_1+q_j$ and so $\cap_{k\in {\mathcal
J}}(Q_1+q_k)\subset Q_1+q_j$ and $\dim S/(Q_1+q_k)\geq 1$ for all
$k\in {\mathcal J}$. Hence $U+I'$ is the intersection of primary
ideals of dimension $\geq 1$, that is $\depth S/(U+I')=1$.  From the
exact sequence
$$0\rightarrow U/(I'\cap U)\rightarrow U'/I'\rightarrow U'/(U+I')\rightarrow 0$$
we get $\depth U/(I'\cap U)=2$. In  the exact sequence
$$0\rightarrow  (U\cap I')/I\rightarrow M\rightarrow U/(I'\cap U) \rightarrow 0$$ the ends are
Cohen-Macaulay modules of dimension $2$ with smaller $e$, the right
one being in fact $(U+I')/I'$ and the left one being $((U\cap
I')+q_1)/q_1$. By induction hypothesis they are clean and so  the
filtration $0\subset (U\cap I')/I\subset M$ can be refined to a
clean one.

Next suppose that $e>1$ and the prime ideals of  $\Ass S/U$ of
dimension $2$ are embedded in the prime ideals of $\Ass S/U$ of
dimension $1$. Suppose that $P_1,P_2\in \Ass S/U$ satisfy
$P_1\subset P_2$, $\dim S/P_1=2$, $\dim S/P_2=1$. As above we may
suppose $P_1=p_1$ and  we show that  $\depth S/(U+I')\geq 1$. Set
$U_1= (Q_1\cap Q_2)+q_1$, $U_2= \cap_{i=3}^t Q_i$ (if $t=2$ take
$U_2=S$). From the exact sequence
$$0\rightarrow U /I\rightarrow U_1/q_1\oplus U_2/I'\rightarrow (U_1+U_2)/(q_1+I') \rightarrow 0$$
we get $(U_1+U_2)/(q_1+I')$ of depth $\geq 1$. Thus if $q_1+q_j$ is
$m$-primary then either $q_1+q_j\supset U_1+U_2$, or there exists
${\mathcal J}\subset [e]$ such that $\dim S/(q_1+q_k)= 1$ for all
$k\in {\mathcal J}$, and $\cap_{k\in {\mathcal J}}(q_1+q_k)\subset
q_1+q_j$. We have $U+I'=(Q_1+I')\cap U'=(\cap_{j=2}^t (Q_1+q_j))\cap
U'$. If $q_1+q_j\supset U_1+U_2$ then we get $Q_1+q_j\supset U+I'$.
In the second case, as above we get  $\cap_{k\in {\mathcal
J}}(Q_1+q_k)\subset Q_1+q_j$ and $\dim S/(Q_1+q_k)= 1$ for all $k\in
{\mathcal J}$.  Hence $\depth S/(U+I')=1$.

Let $L=S\setminus ((\cup_{i=3}^t P_i)\cup (\cup_{j=2}^e p_j))$. Note
that $L^{-1}(U/I)=L^{-1}(U_2/I')$ is Cohen-Macaulay and so $U_2/I'$
is Cohen-Macaulay too. From the exact sequence
$$0\rightarrow U/(I'\cap U)\rightarrow U_2/I'\rightarrow U_2/(U+I')\rightarrow 0$$
we get $\depth U/(I'\cap U)=2$. As above we get $M$ clean using
induction hypothesis on $e$.

Remains to study the case when there are not prime ideals in $\Ass
S/U$ of dimension $2$. Then applying Theorem \ref{pclean} there
exists
 $j\in [e]$ such that $U/(p_j\cap U)$ is a Cohen-Macaulay module of
 dimension $2$ which is clearly clean. From the exact sequence
$$0\rightarrow (U\cap p_j)/I\rightarrow M\rightarrow U/(U\cap p_j)\rightarrow 0$$ we see that
$(U\cap p_j)/ I$ is a Cohen-Macaulay module of dimension 2. Since a
prime ideal of dimension two $p_j$ is contained in $ \Ass S/(U\cap
p_j)$ we get $(U\cap p_j)/ I$ clean as above. Hence the filtration
$0\subset (U\cap p_j)/ I\subset M$ can be refined to a clean one.
\end{proof}

 Next corollary is similar to
\cite[Proposition 1.4]{HSY}.

\begin{Corollary} \label{clean6}
Let $I\subset S$ be a Cohen-Macaulay monomial ideal of dimension
$2$. Then $S/I$ is clean. In particular $\fdepth\ S/I= \sdepth\
S/I=2$ and $I$ is a Stanley ideal.
\end{Corollary}

The following  lemma can be found in \cite{Ra}.
\begin{Lemma}\label{asia}
Let $$0\rightarrow U\rightarrow M\xrightarrow{h} N\rightarrow 0$$ be
an exact sequence of multigraded modules and maps. Then $$\sdepth\
M\geq \min\{\sdepth \ U,\sdepth\ N\}, $$ $$ \fdepth\ M\geq
\min\{\fdepth \ U,\fdepth\ N\}.$$
\end{Lemma}
For the proof note that if $N=\oplus_i y_iK[Z_i]$,
$U=\oplus_jz_jK[T_j]$, $Z_i,T_j\subset \{x_1,\ldots,x_n\}$ are
Stanley decompositions of $N$, respectively $U$, and $y_i'\in M$ are
homogeneous elements with $h(y_i')=y_i$ then $$M=(\oplus_i
y_i'K[Z_i])\oplus(\oplus_jz_jK[T_j])$$ is a Stanley decomposition of
$M$ (the proof of the second inequality is similar).

\begin{Remark}{\em
We believe that the idea of the proof of Theorem \ref{main} could be
used to find nice Stanley decompositions. Some different attempts
were already done in \cite{An}, \cite{HVZ}. We illustrate our idea
on a simple example. Let $S=K[x,y,z,w]$, $I=(x^2,y)\cap
(x,z)\cap(z,w)=(x^2z,x^2w,yz,xyw)$, and $I_1=(x,y)\cap (x,z)\cap
(z,w)=(xz,xw,yz)$. We have $\depth (S/I)=1$ (see \cite[Example
2.7]{AP} for details) and $\depth(S/I_1)=2$ because the associated
simplicial complex of $S/I_1$ is shellable. A clean filtration of
$S/I_1$ is given for example by
$${\mathcal D_1}: \ S/I_1=K[x,y]\oplus w K[y,w]\oplus
zK[w,z]$$ (it corresponds to the partition
$$\Delta=[\emptyset,\{xy\}]\cup [w,\{yw\}]\cup[z,\{wz\}]).$$
The multigraded $S$-module $I_1/I$ is generated by $xz,xw$ and it
has the Stanley decomposition $${\mathcal D_2}: \
I_1/I=xzK[z,w]\oplus xw K[w].$$ We have $\sdepth({\mathcal D_1})=2$,
and $\sdepth({\mathcal D_2})=1$. As $\depth(S/I)=1$ we get from the
exact sequence
$$0\rightarrow I_1/I \rightarrow  S/I \rightarrow  S/I_1 \rightarrow
0$$ that $\depth(I_1/I)=1$.
 Thus $\sdepth ({\mathcal D_2})=\depth (I_1/I)$ and
${\mathcal D_1},{\mathcal D_2}$ induce a Stanley decomposition:
$${\mathcal D}: \ S/I=K[x,y]\oplus w K[y,w]\oplus
zK[w,z]\oplus xzK[z,w]\oplus xw K[w]$$ of $S/I$ with
$\sdepth({\mathcal D})=1=\depth(S/I)$. }
\end{Remark}

Next we will show that  the Stanley depth of some  multigraded
S-modules of dimension 2 is $\geq 1$. An elementary case (a special
case of \cite[Corollary 3.4]{HVZ}) is given by the following lemma:

\begin{Lemma}\label{el}
Let $I\subset K[x,y]$ be a non-zero monomial ideal. If $I$ is not
principal then $\sdepth\ I= \fdepth\ I= 1$.
\end{Lemma}
\begin{proof} Let $\{x^{a_i}y^{b_i}: i\in [s]\}$, $a_i,b_i\in \NN$
be the minimal monomial system of generators of $I$. We may suppose
that $a_1>a_2>\ldots>a_s$ and it follows $b_1<b_2<\ldots <b_s$.
Consider the filtration
$${\mathcal F}: \ \ 0=F_0\subset F_1\subset \ldots\subset F_s=I$$
given by $F_i=F_{i-1}+ (x^{a_{s+1-i}}y^{b_{s+1-i}})$ for $i\geq 1$.
We have $\fdepth\ F_1=2$ because $F_1 $ is free, but
$F_i/F_{i-1}\cong K[x,y]/(y^{b_{s+2-i}-b{s+1-i}})$ has $\fdepth\ =1$
for all $i>1$. Using by recurrence Lemma \ref{asia} we get
$$\sdepth\ I\geq \fdepth\ I\geq 1.$$
If $\sdepth\ I=2$ then $I$ is free by \cite[Lemma 2.9]{Ra} and so
principal, which is false.
\end{proof}

\begin{Remark}{\em The Stanley decomposition associated to a natural
refined filtration of ${\mathcal F}$ is given by
$${\mathcal D}: \ I=x^{a_s}y^{b_s}K[x,y] \oplus (\oplus_{i=1}^{s-1}
\oplus_{j=b_i}^{b_{i+1}-1}  x^{a_i}y^{j }K[x]).$$}
\end{Remark}

Next we extend the above lemma to multigraded modules.
\begin{Proposition} Let $M$ be a torsion-free multigraded
$K[x,y]$-module with $\dim_KM_a\leq 1$ for all $a\in {\ZZ}^2$. If
$M$ is not free then
$$\sdepth\ M= \fdepth\ M= 1.$$
\end{Proposition}
\begin{proof} We have $M/xM\cong K[y]^s\oplus(\oplus_{j=1}^e
K[y]/(y^{r_j}))$ for some non-negative integers $e,s,r_j$ with $e>0$
because $M$ is not free. Since $M$ is multigraded there exist some
$\ZZ^n$-homogeneous elements $(u_i)_{i\in [s]}$ and $(v_j)_{j\in
[e]}$ of $M$ such that
$$M/xM = (\oplus_{i=1}^s {\hat u}_i K[y])\oplus (\oplus_{j=1}^e
\oplus_{t=0}^{r_j-1} {\hat v}_jy^t K)$$ is a Stanley decomposition
of $M/xM$ (we denote by "${\hat z}$" the residue of $z$ modulo $x$).
Let $$N=\sum_{i=1}^s u_i K[x,y]+\sum_{j=1}^e \sum_{t=0}^{r_j-1} v_j
y^t K[x]$$ be a sub-$K[x]$-module of $M$.

We claim that the above sum is direct. Indeed, let $u_if=u_jg$ for
some $i\not =j$. We may suppose $\gcd(f,g)\cong 1$, because
otherwise we may simplify with $\gcd(f,g)$, $M$ being torsion-free.
Then ${\hat u}_if={\hat u}_j g=0$ in $M/xM$ because ${\hat u}_i
K[y]\cap{\hat u}_j K[y]=0$. Thus ${\hat f}={\hat g}=0$ and so
$x|\gcd(f,g)$, which is false. Similarly, we see that $u_iK[x,y]
\cap v_jy^tK[x]=0$ and $ v_jy^tK[x]\cap v_cy^dK[x]=0$ for some
$(j,t)\not =(c,d)$.

Now we see that $yN\subset N$. It is enough to show that
$y^{r_j}v_j\in N$ for all $j$. If $y^{r_j}v_j=0$ there exist nothing
to show. Otherwise, we have $y^{r_j}{\hat v}_j=0$ in $M/xM$ and so
$y^{r_j}v_j\in xM$. Let $p$ be the biggest integer such that
$y^{r_j}v_j\in x^pM$. Since $x^pM$ is multigraded we have
$y^{r_j}v_j=x^pz$ for some $\ZZ^n$-homogeneous element $z\in M$.
Then ${\hat z }={\hat u}_ih$, or ${\hat z }={\hat v}_jq$ for some
$h\in K[y]$, $q\in K$. The cases are similar, let us suppose that
the first holds. Then $y^{r_j}v_j-x^pu_ih\in x^{p+1}M$. If
$y^{r_j}v_j=x^pu_ih\in N$ then we are done. Otherwise, we get
$y^{r_j}v_j\in x^{p+1}M$, because $x^{p+1}M$ is multigraded with
$\dim_K(x^{p+1}M)_a\leq 1$ for all $a\in {\ZZ}^2$. Contradiction!

Thus $N$ is a $K[x,y]$-module and by Nakayama we get $M\subset N$
from $M\subset N+xM$. Hence
$$M=(\oplus_{i=1}^s u_i K[x,y])\oplus(\oplus_{j=1}^e \oplus_{t=0}^{r_j-1} v_j
y^t K[x])$$ is a Stanley decomposition of $M$ and so $\sdepth M\geq
1$.

Consider the filtration
$${\mathcal F}: \ \ 0=F_0\subset F_1\subset \ldots\subset F_{e+1}=M$$
given by $F_1=\sum_{i=1}^s Su_i$, and $F_i=F_{i-1}+Sv_{i-1}$ for
$i>1$. We have $\fdepth\ F_1=2$ because $F_1$ is free and
$F_i/F_{i-1}\cong K[x,y]/(y^{r_{i-1}})$ has $\fdepth\ =1$ for $i>1$.
Then $\fdepth\ M\geq 1$ by Lemma \ref{asia} as in the proof of Lemma
\ref{el}. Note that the above Stanley decomposition is the Stanley
decomposition associated to a natural refined filtration of
${\mathcal F}$. Hence
$$\sdepth\ M\geq\fdepth\ M\geq 1.$$
If $\sdepth\ M=2$ then we obtain $M$ free by \cite[Lemma 2.9]{Ra},
which is false.
\end{proof}

Using the above proposition we get the following theorem inspired by
\cite[Lemma 1.2]{AP1}.

\begin{Theorem}\label{sdepth} Let $M$ be a multigraded $S$-module such
that  $\dim_KM_a\leq 1$ for all $a\in {\ZZ}^n$ and  $\Ass M$ has
only prime ideals of dimension $2$. If $\depth M=1$ then $\sdepth\
M\geq\fdepth \ M= 1$.
\end{Theorem}
\begin{proof}
Let $0=\cap_{i=1}^{r}N_i$, $Q_i=\Ann(M/N_i)$ be the irredundant
primary decomposition of $(0)$ in $M$. Then $\Ass
M=\{P_1,\ldots,P_r\}$, where $P_i=\sqrt{Q_i}$. Apply induction on
$r$. If  $r=1$ and $P_1=Q_1$ then $M$ is torsion-free over $S/P_1$.
Note that $S/P_1$ is a polynomial algebra in two variables. By the
above proposition we get $\sdepth\ M=\fdepth \ M= 1$, because $M$ is
not free over $S/P_1$ since $\depth\ M=1$.

If  $r=1$ but $P_1\not =Q_1$ ($r=1$), then as in the proof of
\cite[Corollary 2.2]{Po_1}) we note that the non-zero factors of the
filtration
$$0=M_s\subset\ldots\subset M_1\subset M_0=M,$$
$M_i=M\cap P_1^i M_{P_1}$ are torsion-free over $S/P_1$ and as above
have $\fdepth \geq 1$. Using several times Lemma \ref{asia} we get
$\fdepth \ M\geq 1$. Since $\sdepth \  M\geq \fdepth M$ it follows
$\sdepth \  M= \fdepth M=1$ because if $\sdepth \  M=2$ then $M$ is
free over $S/Q_1$ by \cite[Lemma 2.9]{Ra} and so $\depth\ M=2$,
which is false.

Now, we suppose that $r>1$. From the exact sequence
$$0\rightarrow N_r \rightarrow M \rightarrow M/N_r\rightarrow 0$$
we have $\sdepth\ M\geq \fdepth\ M\geq \min\{\fdepth\ N_r, \fdepth\
M/N_r\}\geq 1$ by induction hypothesis. As $\fdepth\ M\leq \depth\
M=1$ we are done.
\end{proof}

\section{Case $n=5$}

Let $M$ be a finitely generated multigraded $S$-module. After
\cite{Sc} we consider the {\em dimension filtration}:
$$0\subset D_0(M)\subset D_1(M)\subset \ldots \subset D_t(M)=M$$
of $M$ ($t=\dim M$), which is defined by the property that $D_i(M) $
is the largest submodule of $M$ with $\dim  D_i(M)\leq i$ for $0\leq
i\leq t$ (set $D_{-1}(M)=0$). We have
$\Ass(D_i(M)/D_{i-1}(M))=\{P\in \Ass\ M:\dim S/P=i\}$. $M$ is {\em
sequentially Cohen-Macaulay} if all factors of the dimension
filtration are either 0 or Cohen-Macaulay.

A prime filtration
$${\mathcal F}:\ \ 0=M_0\subset M_1\subset\ldots\subset M_r=M,$$
$M_i/M_{i-1}\cong S/P_i(-a_i)$ is called {\em pretty clean} if for
all $1\leq i<j\leq r$, $P_i\subset P_j$ implies $P_i=P_j$, roughly
speaking "big primes come first". $M$ is called {\em pretty clean}
if it has a pretty clean filtration. After \cite{HP}, $M$ is pretty
clean if and only if all non-zero factors of the dimension
filtration are clean, in which case are Cohen-Macaulay. Thus a
pretty clean $S$-module is sequentially Cohen-Macaulay. Next
proposition extends a result from \cite{AP}.

\begin{Theorem}\label{seq}
 Let $I\subset S$ be a monomial ideal. If $n=5$ then
$S/I$ is pretty clean if and only if $S/I$ is sequentially
Cohen-Macaulay.
\end{Theorem}
\begin{proof}
We have to show only the sufficiency. Suppose that $S/I$ is
sequentially Cohen-Macaulay. Then the factors
$D_i(S/I)/D_{i-1}(S/I)$ are Cohen-Macaulay for all $0\leq i \leq
t=\dim\ S/I$. They are also clean for $0\leq i\leq 1$ by
\cite[Corollary 2.2]{Po_1} (see also Lemma  \ref{clean}). For $i=2$
it is clean by our Theorem \ref{main} applied to $U,U\cap I'$, where
 $D_2(S/I)=U/I$, $I'$ being the intersection of the primary ideals
of dimension $2$ from an irredundant monomial primary decomposition
of $I$ (see \cite{Sc}). As $\Ass(D_3(S/I)/D_2(S/I))$ contains just
prime ideals of codimension 2 we get the cyclic $S$-module
$D_3(S/I)/D_2(S/I)$ clean by \cite[Proposition 1.4]{HSY} (see the
proof of Theorem \ref{5} for more details). If $i>3$ then the
associated prime ideals of the corresponding factors are principal
and the proof is trivial. Hence $S/I$ is pretty clean because all
the non-zero factors of the dimension filtration are clean.
\end{proof}

Next lemma is a variant of \cite[Corollary 1.3]{AP1}.
\begin{Lemma}\label{lem}
Let $I\subset S$ be a monomial ideal having all the associated prime
ideals of height $2$. If $n=5$ then $I$ is a Stanley ideal and
$$2\leq \sdepth(S/I)=\fdepth(S/I)=\depth(S/I)\leq 3.$$
\end{Lemma}

\begin{Theorem} \label{5} Let $I\subset S$ be a  monomial ideal. If
$n=5$ then $I$ is a Stanley ideal and $\sdepth\ S/I\geq
\fdepth(S/I)=\depth(S/I)$.
\end{Theorem}
\begin{proof}
We follow the proof given for $n=4$  in \cite[Proposition 1.4]{AP1}.
Let
$$0\subset D_0(S/I)\subset D_1(S/I)\subset \ldots \subset
D_4(S/I)=S/I$$ be the dimension filtration of $S/I$. Set
$E_i=D_i(S/I)/D_{i-1}(S/I)$ for $0\leq i\leq 4$. Let $I=\cap_{i=0}^4
\cap_{j=1}^{s_i} Q_{ij}$, $\dim Q_{ij}=i$ be an irredundant monomial
primary decomposition of $I$. Set $Q_i=\cap_{j=1}^{s_i} Q_{ij}$.
After \cite{Sc} we have $D_k(S/I)=(\cap_{k<i\leq 4} Q_i)/I$, $-1\leq
k\leq 3$. Clearly $Q_4/I=D_3(S/I)\subset S/I$ is clean because
$\Ass(S/Q_4)$ contains only prime ideals of height one, which are
principal. It follows $\fdepth((S/I)/D_3(S/I))=4$. Moreover $Q_4$ is
principal, let us say $Q_4=uS$ for some monomial $u$. Then $E_3\cong
S/(Q_3:u)$. Note that $\Ass S/(Q_3:u)$ contains only prime ideals of
codimension 2 and so  by Lemma \ref{lem} $(Q_3:u)$ is a Stanley
ideal and $\fdepth \ ((S/I)/D_2(S/I))= \depth ((S/I)/D_2(S/I))$. If
$D_2(S/I)=0$ we are done.

Suppose $s_2\not =0$. Then the associated prime ideals of $E_2$ have
dimension 2. If $E_2 $ is Cohen-Macaulay then $\fdepth\ E_2=\sdepth\
E_2=2$ by Theorem \ref{main} applied to $U=\cap_{2<i\leq
4}Q_i\supset U\cap Q_2$. If $\depth E_2=1$ (so $\depth S/I\leq 1$)
then by Theorem \ref{sdepth} we have $\fdepth\ E_2= 1$. Thus
$\fdepth \ ((S/I)/D_1(S/I))= \depth S/I$. If $D_1(S/I)=0$ we are
done.

Suppose $s_1\not =0$. Then $E_1$ is clean by \cite[Corollary
2.2]{Po_1}). If $\depth S/I=1$ then $  \depth E_1=1$ and we get
$\fdepth\ E_1=\sdepth\ E_1=1$. Thus $\fdepth(S/I)/D_0(S/I)=\fdepth\
S/Q_1=1$. If $\depth S/I=0$ then there is nothing to be shown.
\end{proof}

The inequality from the above theorem could be strict, as shows the
following:
\begin{Example}{\em Let $n=5$ and $I=(x_1,x_2)\cap
(x_3,x_4,x_5)\subset S$. Then $\sdepth\ S/I=2>\depth\ S/I=\fdepth\
S/I=1$.}
\end{Example}

\end{document}